\documentclass[12pt]{article}
\usepackage{amsmath,amsthm, amssymb, amsfonts,accents,nccmath}
\usepackage{enumitem}
\usepackage{array}
\usepackage{lipsum}
\usepackage{mathtools}
\usepackage[toc,page]{appendix}
\usepackage[english]{babel}
\usepackage{mathrsfs}
\usepackage{dsfont}
\usepackage{booktabs}
\usepackage[linesnumbered,commentsnumbered,ruled,vlined]{algorithm2e}
\usepackage{tikz}
\usepackage{tikz-network}
\usepackage[utf8]{inputenc}
\usepackage[english]{babel}

\DeclareMathOperator{\diag}{diag}
\usepackage{caption}
\usepackage{subcaption}

\newcommand\restr[2]{{
  \left.\kern-\nulldelimiterspace 
  #1 
  \vphantom{\big|} 
  \right|_{#2} 
  }}

\usepackage{hyperref}
\hypersetup{
    colorlinks=true,
    linkcolor=blue,
    filecolor=darkmagenta,      
    urlcolor=cyan,
    citecolor=blue
} 
\makeatletter
\def\BState{\State\hskip-\ALG@thistlm}
\makeatother
\numberwithin{equation}{section}

\newtheorem{theorem}{Theorem}[section]

\newtheorem{lem}[theorem]{Lemma}

\title{ On the maximum $A_{\alpha}$-spectral radius of unicyclic and bicyclic graphs with fixed girth or fixed number of pendant vertices}
\author{
	Joyentanuj Das\thanks{Department of Applied Mathematics, National Sun Yat-sen University, Kaohsiung City - 804, Taiwan (R.O.C.). Email: joyentanuj@gmail.com}
	\and
	Iswar Mahato\thanks{Department of Mathematics, Indian Institute of Technology Bombay, Mumbai 400076, India, \newline Email: iswarmahato02@gmail.com, iswar@math.iitb.ac.in}
}
\date{}

\begin{document}

\maketitle

\begin{abstract}
For a connected graph $G$, let $A(G)$ be the adjacency matrix of $G$ and  $D(G)$ be the diagonal matrix of the degrees of the vertices in $G$. The $A_{\alpha}$-matrix of $G$ is defined as 
\begin{align*}
A_\alpha (G) = \alpha D(G) + (1-\alpha) A(G) \quad \text{for any $\alpha \in [0,1]$}.     
\end{align*}
 The largest eigenvalue of $A_{\alpha}(G)$ is called the $A_{\alpha}$-spectral radius of $G$. In this article, we characterize the graphs with maximum $A_{\alpha}$-spectral radius among the class of unicyclic and bicyclic graphs of order $n$ with fixed girth $g$. Also, we identify the unique graphs with maximum $A_{\alpha}$-spectral radius among the class of unicyclic and bicyclic graphs of order $n$ with $k$ pendant vertices.
\end{abstract}



\section{Introduction}
 Throughout the article, we only consider the finite, simple and connected graphs. Let $G=(V(G),E(G))$ be a graph with vertex set $V(G)=\{v_1,v_2,\hdots,v_n\}$ and edge set $E(G)=\{e_1,e_2,\hdots,e_m\}$. The cardinality of $V(G)$ and $E(G)$ are the \textit{order} and the \textit{size} of $G$, respectively. If $u,v\in V(G)$ are adjacent, then we write $u\sim v$, otherwise $u\nsim v$. The adjacent vertices are called \textit{neighbours} and the set of neighbours of a vertex $v$ in $G$ is denoted by $N_G(v)$. The \textit{degree} of a vertex $v$, denoted by $d_G(v)$ (or simply $d(v)$), is the number of adjacent vertices of $v$ in $G$. A vertex of degree $1$ is called the \textit{pendant vertex} of $G$. A \textit{pendant edge} is an edge incident to a pendant vertex. A graph $H$ is said to be a \textit{subgraph} of $G$ if $V(H) \subseteq V(G)$ and $E(H) \subseteq E(G)$. The \textit{distance} $d(u,v)$ between two vertices $u,v\in V(G)$ is the length of a shortest path connecting $u$ and $v$ in $G$.
 
 The \textit{adjacency matrix} of a graph $G$ on $n$ vertices, denoted by $A(G)$, is the $n \times n$ matrix with its rows and columns are indexed by the vertex set of $G$ and $A(G)_{ij}=1$, if the vertices $v_i$ and $v_j$ are adjacent and $A(G)_{ij}=0$ otherwise. Let $D(G)=\diag(d_1,d_2,\hdots,d_n)$ be the degree matrix of $G$, where $d_i$ is the degree of the vertex $v_i\in V(G)$. The \textit{signless Laplacian matrix} of $G$ is defined as $Q(G)=D(G)+A(G)$. In \cite{Niki0}, Nikiforov defined the matrix $A_\alpha(G)$ as 
\begin{align*}
A_\alpha (G) = \alpha D(G) + (1-\alpha) A(G) \quad \text{for any $\alpha \in [0,1]$}.     
\end{align*}

Note that if $\alpha = 0$, then $A_\alpha (G) = A(G)$ and if $\alpha = \frac12 $, then $A_\alpha (G) = \frac12 Q(G)$. Thus, $A_\alpha (G)$ generalizes both the adjacency matrix and the signless Laplacian matrix of $G$. The largest eigenvalue of $A_\alpha (G)$ is called the $A_{\alpha}$-spectral radius of $G$, which is denoted by $\rho_\alpha(G)$. For a connected graph $G$, $A_{\alpha}(G)$ is always irreducible and hence by the Perron-Frobenius theorem the multiplicity of $\rho_\alpha(G)$ is one and there exists a unique positive (entry-wise) unit eigenvector corresponding to $\rho_\alpha(G)$, which is called the \textit{Perron vector} of $A_{\alpha}(G)$.

Spectral extremal problems are one of the interesting and well-studied problems in spectral graph theory. For the adjacency matrix and the signless Laplacian matrix of a graph, the extremal problems have been extensively studied. In \cite{Guo}, Guo determined the graphs with maximum spectral radius among all unicyclic and bicyclic graphs with $n$ vertices and $k$ pendant vertices. Zhai et al. \cite{Zhai2009} identified the unique graphs with maximum spectral radius among all bicyclic graphs on $n$ vertices with fixed girth. Liu et al. \cite{Liu} characterized the graphs with the largest signless Laplacian spectral radius among all unicyclic and bicyclic graphs with $n$ vertices and $k$ pendant vertices. Li et al. \cite{Li2011signless} determined the unique graphs with maximum signless Laplacian spectral radius among all unicyclic and bicyclic graphs on $n$ vertices with fixed girth. For more details about the extremal problems of the spectral radius and the signless Laplacian spectral radius of graphs, see \cite{Geng,Li2013,Li2011tricyclic,Liu2012signless,Wu} and the references therein.

In recent years, the study of extremal problems for the $A_{\alpha}$-spectral radius of graphs have received significant importance and attracted the attention of researchers. In~\cite{Xue}, Xue et al. determined the unique graphs with maximum $A_{\alpha}$-spectral radius among all connected graphs with fixed diameter; Lin et al.~\cite{Lin} characterized the extremal trees with given order and matching number having the maximum $A_{\alpha}$-spectral radius; Li et al.~\cite{Li} characterized the extremal graphs with maximum $A_{\alpha}$-spectral radius among all trees and unicyclic graphs with given degree sequence; Rojo \cite{Rojo} characterized the extremal trees with maximum $A_{\alpha}$-spectral radius among all trees of order $n$ with $k$ pendant vertices; Li and Sun \cite{Sun} determined the extremal graphs with  maximum $A_{\alpha}$-spectral radius among the class of bipartite graphs with fixed order and size. In \cite{Sun2}, the authors characterized the extremal graphs with given order and size having the maximum $A_{\alpha}$-spectral radius.  Wen et al.~\cite{Wen} studied the extremal problems for $A_{\alpha}$-spectral radius of bicyclic graphs with prescribed degree sequence and determined the extremal graphs. Wang et al. \cite{Shaohui} characterized the cacti of fixed order and fixed number of pendant vertices having the maximum $A_{\alpha}$-spectral radius. Very recently, Wang et al.~\cite{Wang} characterized the extremal unicyclic and bicyclic graphs with fixed diameter, whose $A_{\alpha}$-spectral radii are maximum. Chen et al. \cite{Chen} determined the graphs with maximum $A_{\alpha}$-spectral radius among  the graphs of given size and girth. For more advances on the extremal problems of the $A_{\alpha}$-spectral radius of graphs, see the articles \cite{Chang2023,Chen2022,Feng2022,Lou2023,Yuan} and the references therein. Motivated by these works, in this article, we study the maximization problems for the $A_{\alpha}$-spectral radius of unicyclic and bicylic graphs with fixed girth or fixed number pendant vertices and determine the maximal graphs. 

The article is organized as follows: In section $2$, we define the useful notations and collect some preliminary results. In section $3$, we characterize the unicyclic graphs with maximum $A_{\alpha}$-spectral radius among the class of unicyclic graphs on $n$ vertices with girth $g$ or $k$ pendant vertices. In section $4$, we determine the bicyclic graphs for which the $A_{\alpha}$-spectral radius is maximum among all bicyclic graphs of order $n$ with girth $g$ or $k$ pendant vertices.


\section{Notations and Preliminaries}
In this section, we define the useful notations and collect some known results for the $A_{\alpha}$-spectral radius of graphs, which will be used to prove our results in the subsequent sections.

For integers $n$ and $g$, let $\mathcal{U}_n(g)$ and $\mathcal{B}_n(g)$ denote the class of all connected unicyclic and bicyclic graphs of order $n$ with girth $g$, respectively. Let $\mathscr{U}_n(k)$ and $\mathscr{B}_n(k)$ be the set of all connected unicyclic and bicyclic graphs of order $n$ with $k$ pendant vertices, respectively. For a real number $x$, let $\lceil x \rceil$ denote the least integer greater than or equal to $x$ and $\lfloor x \rfloor$ denote the greatest integer less than or equal to $x$.

Let $C_n$ and $P_n$ denote the cycle and the path on $n$ vertices, respectively. For a graph $G$, let $G-u$ or $G-uv$ denote the graph obtained from $G$ by deleting the vertex $u\in V(G)$ or the edge $uv\in E(G)$. Similarly, $G+uv$ denote the graph obtained from $G$ by adding the edge $uv \notin E(G)$, where $u,v\in V(G)$. A path $P_{k+1}=v_0v_1v_2\hdots v_k$ in $G$ is called a pendant path from a vertex $v_0$ if $d(v_1)=\hdots =d(v_{k-1})=2$ and $d(v_k)=1$. A path $P_{k+1}=v_0v_1v_2\hdots v_k$ is said to be an internal path if $d(v_0) \ge 3$, $d(v_k) \ge 3$ and $d(v_1) = \hdots = d(v_{k-1}) = 2$. $k$ paths $P_{l_1+1},P_{l_2+1},\hdots,P_{l_k+1}$ are said to have almost equal lengths if $|l_i-l_j|\leq 1$ for $1\leq i,j\leq k$. Now, we collect some preliminary results for the $A_{\alpha}$-spectral radius of graphs.

\begin{lem}[{\cite[Proposition 14]{Niki0}}]\label{lem:subgraph}
Let $\alpha \in [0,1)$. If $H$ is a proper subgraph of a graph $G$, then $\rho_\alpha(H)<\rho_\alpha(G)$.
\end{lem}

\begin{lem}[{\cite[Corollary 19]{Niki0}}]\label{lem:spec}
If $\alpha \in [0,1)$ and $G$ is a graph on $n$ vertices and $m$ edges, then $$\rho_\alpha(G) \ge \frac{2m}{n}.$$
The equality holds if and only if $G$ is a regular graph.
\end{lem}

\begin{lem}[{\cite[Corollary 13]{Niki0}}]\label{lem:spec-Delta}
 Let $G$ be a connected graph with maximum degree $\Delta$. Then
$$\rho_\alpha(G) \ge
\begin{cases}
	\text{$\alpha(\Delta+1)$} & \quad\text{if $\alpha \in [0,\frac{1}{2}]$,}\\
	\text{$\alpha \Delta + \frac{(1-\alpha)^2}{\alpha}$} & \quad\text{if $\alpha \in [\frac{1}{2},1)$.}
\end{cases}$$
\end{lem}

\begin{lem}[{\cite[Lemma 2.2]{Xue}}]\label{lem:edge}
Let $\alpha \in [0,1)$ and $G$ be a connected graph with $u,v \in V(G)$ such that $N \subseteq N(v) \setminus (N(u) \cup \{u\})$. Let $G' = G - \{vw : w \in N\} + \{uw : w \in N\}$ and $\mathbf{x}=(x_1,\hdots,x_n)^t$ be the Perron vector of $A_{\alpha}(G)$ corresponding to $\rho_\alpha(G)$. If $N\neq \phi$ and  $x_u \ge x_v$, then $\rho_\alpha(G)<\rho_\alpha(G')$.
\end{lem}

\begin{lem}[{\cite[Lemma 1.1]{Li}}]\label{lem:subdiv}
Let $\alpha \in [0,1)$ and $G$ be a connected graph. If $uv$ is an edge on some internal path of $G$ and $G_{uv}$ is the graph obtained from $G$ by subdivision of the edge $uv$ into the edges $uw$ and $wv$, then $\rho_\alpha(G_{uv}) < \rho_\alpha(G)$.
\end{lem}

\begin{lem}[{\cite[Theorem 3.2]{Guo2020}}]\label{lem:path}
Let $u,v$ be two vertices of a connected graph $G$ with degree at least $2$. Let $G_{u,v} (k,l)$ be the graph obtained from $G$ by attaching the pendant paths $P_k$ to $u$ and $P_l$ to $v$, respectively. If $k-l \ge 2$ and $uv\in E(G)$, then for any $\alpha \in [0,1)$
$$\rho_\alpha(G_{u,v} (k,l)) < \rho_\alpha(G_{u,v} (k-1,l+1)).$$
\end{lem}

If a tree $T$ is attached to a vertex $v$ of a graph $G$, then the vertex $v$ is said to be the root of $T$ in $G$.

\begin{lem}[{\cite[Proposition 7]{Niki}}]\label{lem:path-pendant}
Let $\alpha \in [0,1)$ and $G$ be a graph with $\rho_\alpha(G) > 2$. Let $P = v_1v_2 \hdots v_{r+1}$ be a pendant path in $G$ with root $v_1$. Let $\mathbf{x}=(x_1,\hdots,x_n)^t$ be the Perron vector of $A_{\alpha}(G)$ corresponding to $\rho_\alpha(G)$. Then $x_1 > \hdots > x_{r+1}$.
\end{lem}

For two graphs $G_1$ and $G_2$ with $u\in V(G_1)$ and $v\in V(G_2)$, the coalescence of $G_1$ and $G_2$, denoted by $G_1(u,v)G_2$, is the graph obtained by identifying the vertices $u$ and $v$.

\begin{lem}[{\cite[Lemma 3.2]{Wang}}]\label{lem:coal}
 Let $G_1$ and $G_2$ be two graphs with $u\in V(G_2)$ and $v_1,v_2\in V(G_1)$ such that $N_{G_1}(v_1)\setminus \{v_2\}\subset N_{G_1}(v_2)\setminus \{v_1\}$. If $H_1=G_1(v_1,u)G_2$ and $H_2=G_1(v_2,u)G_2$, then $\rho_\alpha(H_2)>\rho_\alpha(H_1)$.   
\end{lem}

For any vertex $u \in V(G)$ in a graph $G$, let $$m(u) = \frac{1}{d(u)} \sum_{v:v \sim u} d(v).$$

\begin{lem}[{\cite[Proposition 20]{Niki0}}]\label{lem:d(u)-m(u)}
If $\alpha \in [0,1)$ and $G$ is a graph with no isolated vertices, then 
$$\rho_\alpha(G) \le \max_{u \in V(G)} \{\alpha d(u) + (1-\alpha)m(u)\}.$$ 
 If $\alpha \in (\frac{1}{2},1)$ and $G$ is connected, then the equality holds if and only if $G$ is regular.
\end{lem}

\begin{lem}[{\cite[Corollary 21]{Niki0}}]\label{lem:d(u)-d(v)}
If $\alpha \in [0,1)$ and $G$ is a connected graph, then $$\rho_\alpha(G) \le \max_{uv \in E(G)} \{\alpha d(u) + (1-\alpha) d(v)\}.$$ 
 The equality holds if and only if $G$ is regular or $G$ is a semi-regular bipartite graph with $\alpha = \frac{1}{2}$.
\end{lem}


\section{Extremal unicyclic graphs on $n$ vertices with girth $g$ or $k$ pendant vertices}

In this section, we characterize the extremal graphs with maximum $A_\alpha$-spectral radius among $\mathcal{U}_n(g)$ and $\mathscr{U}_n(k)$, respectively. First, let us prove the following graph transformation. 

\begin{lem}\label{lem:path2}
Let $G$ be a graph with a cycle $C_t$, in which a path $P_{l+1}$ of length $l\geq 2$ is attached to a vertex $u$ of $C_t$. Let $G^*$ be the graph obtained from $G$ by replacing the path $P_{l+1}$ by $l$ pendant edges at $u$. Then $\rho_\alpha(G)<\rho_\alpha(G^*)$.
\end{lem}

\begin{proof}
	Let $P_{l+1}= u_0 u_1 \hdots u_l$ be a path attached to the vertex $u$ of $C_t$, and hence $u=u_0$. Let $x_0,x_1,\hdots,x_{l}$ be the entries of a unit Perron vector of $A_\alpha(G)$ corresponding to the vertices $u_0,u_1,\hdots,u_{l}$. Since $l \geq 2$, $u_{l-1} \neq u$ and hence by Lemma \ref{lem:path-pendant}, we have $x_0 > x_{l-1}$. Therefore, by applying Lemma \ref{lem:edge}, we have a graph $G_1=G-u_{l-1}u_l+uu_l$ with $\rho_\alpha(G)<\rho_\alpha(G_1)$. By repeatedly applying this transformation, we get the graph $G^*$, in which the pendant path of length $l$ is replaced by $l$ pendant edges at $u$ and $\rho_\alpha(G)<\rho_\alpha(G^*)$. 
 \end{proof}

 Let $U_n(g)$ be the unicyclic graph with a cycle of length $g$ and $n-g$ pendant vertices attached to a single vertex of the cycle. In the following theorem, we prove that the unicyclic graph $U_n(g)$ has maximum $A_\alpha$-spectral radius in $\mathcal{U}_n(g)$. 

\begin{theorem}\label{thm:uni-girth}
 Let $G$ be a unicyclic graph on $n$ vertices with girth $g \ge 3$. Then $$\rho_\alpha(G)\leq \rho_\alpha(U_n(g)), $$
 and the equality holds if and only if $G \cong U_n(g)$.
\end{theorem}

\begin{proof}
	Let $G \in \mathcal{U}_n(g)$ be a graph with the maximal $A_\alpha$-spectral radius and  $V(G) = \{v_1,v_2,\hdots,v_n\}$. Let $\mathbf{x} = (x_1,x_2,\hdots,x_n)^t$ be the unit Perron vector of $A_\alpha(G)$, where $x_i$ corresponds to the vertex $v_i$. Let $C_g$ be the unique cycle of length $g$ in $G$. Without loss of generality, we assume that $V(C_g) = \{v_1,v_2,\hdots, v_g\}$.
	
    First, we claim that the graph $G$ has trees attached to a single vertex of the cycle $C_g$. Suppose that there are two trees $T_1$ and $T_2$ incident to the vertices $v_1$ and $v_2$, respectively. Without loss of generality, we assume that $x_1 \ge x_2$. Since $G$ is unicyclic graph, there exists a vertex $u \in N(v_2)\setminus \{v_1\}$ but $u\notin N(v_1)$. If $G_1=G-uv_2+uv_1$, then $G_1 \in  \mathcal{U}_n(g)$ and by Lemma \ref{lem:edge}, we have $\rho_\alpha(G) < \rho_\alpha(G_1)$, which contradicts to the fact that $G$ has maximal $A_\alpha$-spectral radius in $\mathcal{U}_n(g)$. Therefore, all trees in $G$ are incident to a single vertex, say $v_1$ of the cycle $C_g$. Next, we claim that each of the trees are paths. If not, then there exists a tree $T$ with a vertex $v$ such that $d(v) \ge 3$ and we choose the vertex $v$ such that $d(v_1,v)$ is minimum. Now, consider the following two cases:
	
	\textbf{Case 1.} Let $x_1 \ge x_v$. Since $d(v) \ge 3$ there exists a vertex $w \in N(v)$ such that $d(w,v_1) > d(v,v_1)$. Now, consider the graph $G_2=G-vw+v_1w$, then $G_2 \in  \mathcal{U}_n(g)$ and by Lemma \ref{lem:edge}, we have $\rho_\alpha(G) < \rho_\alpha(G_2) $, which is a contradiction.
 
	\textbf{Case 2.} Let $x_1 < x_v$. For this case, we divide the proof into the following two subcases:
	
   \textbf{Subcase I.} Let $d(v_1,v) = 1$. Since $G$ is a unicyclic graph, $v_2$ is not adjacent to $v$. If $G_3=G-v_1v_2+vv_2$, then  $G_3 \in \mathcal{U}_n(g+1)$ and by Lemma \ref{lem:edge}, it follows that $\rho_\alpha(G)<\rho_\alpha(G_3)$. Now, consider $G_4=G_3-v_1v_2-v_2v_3+vv_3-v_2$, then $G_4 \in \mathcal{U}_{n-1}(g)$. Since $d(v) \ge 3$, $vv_3\hdots v_tv_1v$ is an internal path of $G_3$ and hence by Lemma \ref{lem:subdiv}, we have  $\rho_\alpha(G_3) < \rho_\alpha(G_4)$. Now, construct a graph $G_5$ from $G_4$ by adding the deleted vertex $v_2$ with a pendant vertex of $G_4$. Then $G_5 \in \mathcal{U}_{n}(g)$, and by Lemma \ref{lem:subgraph}, it follows that $\rho_\alpha(G_4) < \rho_\alpha(G_5)$. Thus, there exists a graph $G_5 \in \mathcal{U}_{n}(g)$ such that $\rho_\alpha(G) < \rho_\alpha(G_5)$, which is a contradiction.
	
	\textbf{Subcase II.} Let $d(v_1,v) \ge 2$. Let $P$ be the unique path from $v_1$ to $v$. Since the length of the path $P$ is at least $2$, there exists a vertex $w$ on the path $P$ such that $w \neq v,v_1$ and $w \sim v$. Clearly, $d(w) = 2$ and let $v_l$ be the neighbour of $w$ other than $v$ (it is possible that $v_l=v_1$). Let $G_6=G-wv-v_lw-w$. Since $P = v_1  \hdots v_lwv$ is an internal path of $G$, by Lemma \ref{lem:subdiv}, it follows that $\rho_\alpha(G) < \rho_\alpha(G_6)$. Now, construct a graph $G_7$ from $G_6$ by adding the deleted vertex $w$ with a pendant vertex of $G_6$. Then, it follows by Lemma \ref{lem:subgraph} that $\rho_\alpha(G_6) < \rho_\alpha(G_7)$. Thus, there exists a graph $G_7 \in \mathcal{U}_{n}(g)$ such that $\rho_\alpha(G) < \rho_\alpha(G_7)$, which is a contradiction.
	
   Thus, the maximal graph $G$ has only paths attached to a single vertex of the cycle $C_g$. Now, by Lemma \ref{lem:path2}, we can conclude that the paths are pendant edges, and hence $G \cong U_n(g)$.
\end{proof}

Now, we consider the connected unicyclic graphs of order $n$  with $k$ pendant vertices and characterize the extremal graphs with maximum $A_\alpha$-spectral radius. Let $\mathscr{U}_n(t,k)$ be the set of all connected unicyclic graphs of order $n$ with a unique cycle of length $t$ and $k$ pendant vertices. Let $U_n(t,k)$ denote the unicyclic graph of order $n$ with a unique cycle $C_t$ of length $t$ and there are $k$ paths of almost equal lengths attached to a single vertex of $C_t$. In the following lemma, we prove that $U_n(t,k)$ has maximum $A_\alpha$-spectral radius in $\mathscr{U}_n(t,k)$.

\begin{lem}\label{lem:uni}
Let $t,k$ be integers such that $t \ge 3$ and $1 \le k \le n-t$. If $G \in \mathscr{U}_n(t,k)$, then	
 $$\rho_\alpha(G) \leq \rho_\alpha(U_n(t,k))$$
 with equality if and only if $G \cong U_n(t,k)$.
\end{lem}

\begin{proof}
	Let $G$ be a graph with maximal $A_\alpha$-spectral radius in $\mathscr{U}_n(t,k)$. Therefore, by Theorem~\ref{thm:uni-girth}, we can conclude that $G$ has a cycle $C_t$ of length $t$ and $k$ pendant paths are attached to a single vertex of $C_t$. Again, by Lemma~\ref{lem:path}, all the $k$ paths have almost equal lengths. Thus, $G \cong U_n(t,k)$.
\end{proof}

\begin{lem}\label{lem:uni2}
	Let $t,k$ be integers such that $t \ge 4$ and $1 \le k \le n-t$. Then $$\rho_\alpha(U_n(t,k))<\rho_\alpha(U_n(t-1,k)).$$
\end{lem}

\begin{proof}
	Let $C_t = v_1v_2\hdots v_tv_1$ be the unique cycle of length $t$ in $U_n(t,k)$. Without loss of generality, we assume that the $k$ paths of almost equal lengths are attached to the vertex $v_1$. Since $d(v_1) \ge 3$,  $v_1v_2\hdots v_tv_1$ is an internal path of $U_{n-1}(t-1,k)$ and hence by Lemma \ref{lem:subdiv}, we have $\rho_\alpha(U_n(t,k)) < \rho_\alpha(U_{n-1}(t-1,k))$. Let $G_1$ be a graph constructed from $U_{n-1}(t-1,k)$ by attaching an isolated vertex to a pendant vertex of $U_{n-1}(t-1,k)$. Then $G_1\in \mathscr{U}_n(t-1,k)$ and by Lemma \ref{lem:subgraph}, we have $\rho_\alpha(U_{n-1}(t-1,k)) \le \rho_\alpha(G_1)$. Now, by Lemma \ref{lem:uni}, it follows that $\rho_\alpha(G_1) \le \rho_\alpha(U_n(t-1,k))$. This completes the proof.
\end{proof}

Now, we characterize the unicyclic graphs with maximum $A_\alpha$-spectral radius in $\mathscr{U}_n(k)$.

\begin{theorem}
Let $G$ be a unicyclic graph of order $n$ with $k$ pendant vertices, where $1 \le k \le n-3$. Then $$\rho_\alpha(G)\leq \rho_\alpha(U_n(3,k))$$ 
 and the equality holds if and only if $G \cong U_n(3,k)$.
\end{theorem}

\begin{proof}
	Since $G \in \mathscr{U}_n(k)$, there exists an integer $t \ge 3$ such that $G \in \mathscr{U}_n(t,k)$. By Lemma~\ref{lem:uni}, we have $\rho_\alpha(G)\leq \rho_\alpha(U_n(t,k))$. Now, by repeatedly applying the Lemma \ref{lem:uni2} to $U_n(t,k)$, one eventually reaches to the graph $U_n(3,k)$ and $\rho_\alpha(U_n(t,k))\leq \rho_\alpha(U_n(3,k))$. Thus, $\rho_\alpha(G)\leq \rho_\alpha(U_n(3,k))$.
\end{proof}


\section{Bicyclic graphs with $n$ vertices and girth $g$ or $k$ pendant vertices}

In this section, we characterize the extremal bicyclic graphs with maximum $A_\alpha$-spectral radius among $\mathcal{B}_n(g)$ and $\mathscr{B}_n(k)$, respectively. For a bicyclic graph $G$, let $B(G)$ be the minimal bicyclic subgraph of $G$. That is, $B(G)$ is the unique bicyclic subgraph of $G$ without pendant vertices, and $G$ can be derived from $B(G)$ by attaching trees to some vertices of $B(G)$. 

\subsection{Extremal bicyclic graphs in $\mathcal{B}_n(g)$}
In this subsection, we determine the unique graphs with maximum $A_\alpha$-spectral radius among the class of all bicyclic graphs of order $n$ with fixed girth $g$. Let $\mathcal{B}_n^{1}(g)\in \mathcal{B}_n(g)$ such that the intersection of the two cycles is at most one vertex, and let $\mathcal{B}_n^{2}(g)\in \mathcal{B}_n(g)$ such that the intersection of the two cycles is at least one edge. Note that $\mathcal{B}_n(g) = \mathcal{B}_n^{1}(g) \cup \mathcal{B}_n^{2}(g)$. 

Let $B_n^1(p,t,q)\in \mathcal{B}_n^{1}(g)$ be a class of bicyclic graphs with the cycles of length $p,q$ and there is a path of length $t-1$ between the cycles, where $t \ge 1$. Let $B_n^1(g)$ denote the class of bicyclic graphs, where both the cycles are of length $g$ with exactly one common vertex which is adjacent to $n-2g+1$ pendant vertices.

\begin{theorem}\label{thm:bicyclic-girth}
	Let $n,g$ be integers with $3 \le g \le \frac{n+1}{2}$. If $G \in  \mathcal{B}_n^{1}(g)$, then $$\rho_\alpha(G) \leq \rho_\alpha(B_n^1(g))$$ and the equality holds if and only if $G \cong B_n^1(g)$.
\end{theorem}

\begin{proof}
Let $G \in \mathcal{B}_n^{1}(g)$ be a graph such that the 
$A_\alpha$-spectral radius of $G$ is as large as possible. Therefore, $G \in B_n^1(p,t,q)$ for some $t \ge 1$ and $g=p\leq q$. Let $V(G) = \{v_1,v_2,\hdots,v_n\}$ be the vertex set of $G$ and let $\mathbf{x} = (x_1,\hdots,x_n)^t$ be the Perron vector of $A_\alpha(G)$, where $x_i$ corresponds to the vertex $v_i$. 
	First, we claim that $t = 1$, that is, in the maximal graph $G$ the two cycles share exactly one common vertex. If not, then there exists a path of length $t-1$ with starting vertex $v_1 \in C_g$ and ending vertex $v_t \in C_q$. Without loss of generality, we assume that $x_1 \ge x_t$. Now, consider a graph $G_1$ such that $G_1=G-v_tv_{t+1}+v_1v_{t+1}$, where $v_{t+1}$ is the neighbour of $v_t$ in $C_q$. Then, by Lemma \ref{lem:edge}, we have $\rho_\alpha(G) > \rho_\alpha(G_1)$, which is a contradiction.
	
	Next, we claim that if there is any tree $T$ in $G$, then $T$ is attached to $v_1$, the common vertex of the cycles $C_g$ and $C_q$. If not, then there exists a vertex $v_i$ in $C_g$ such that a tree $T_i$ is attached to $v_i$. If $x_1 \ge x_i$, then consider $N = N(v_i) \setminus V(C_g)$ and by using Lemma~\ref{lem:edge}, we get a contradiction. Similarly, if $x_i > x_1$, then choose $N = N(v_1) \setminus V(C_g)$ and by applying Lemma~\ref{lem:edge}, we get a contradiction. 
	
	Now, we claim that all trees attached to $v_1$ are paths. If not, there exists a vertex $v$ on a tree such that $d(v) \ge 3$. If $x_1 \ge x_v$, then the proof is same as the proof of Case $1$ of Lemma~\ref{lem:uni}. Suppose that $x_v > x_1$. Let $v_2,v_3$ be the vertices adjacent to $v_1$ in $C_p$ and $v_4$ be the vertex adjacent to $v_1$ on the path from $v_1$ to $v$. Then, consider $N = N(v_1) \setminus \{v_2,v_3\}$ and applying Lemma~\ref{lem:edge}, we get a contradiction.
	
	Thus, the maximal graph $G$ has two cycles $C_g$ and $C_q$ with exactly one common vertex $v_1$ and there are some paths attached to the vertex $v_1$. Again, by the same argument as in Lemma~\ref{lem:uni2}, we can reduce the size of the larger cycle $C_q$ to $g$. Finally, by using Lemma~\ref{lem:path-pendant}, we can conclude that all the paths attached to $v_1$ are pendant edges. Thus, $G \cong B_n^1(g)$.
\end{proof}

Let $P_{p+1}, P_{q+1}, P_{r+1}$ be three vertex disjoint paths with $p,q,r \ge 1$ and at most one of $p,q,r$ is $1$. Let $C(p,q,r)$ be the graph obtained by identifying the end vertices of the paths $P_{p+1}, P_{q+1}$ and $P_{r+1}$ on both sides. Let $C_{p,q,r}^t$ be the bicyclic graph obtained from $C(p,q,r)$ by attaching $t$ pendant vertices to a vertex of degree $3$.

\begin{theorem}\label{thm:bi-2-unique}
	Let $n,g$ be integers such that $\lceil \frac{3g}{2} \rceil - 1 \le n$ and $G^*$ be the graph with maximal $A_\alpha$-spectral radius in $\mathcal{B}_n^2(g)$ for $\frac{1}{2}\leq \alpha <1$. Then $G^*$ is obtained from $C(\lfloor \frac{g}{2} \rfloor, \lceil \frac{g}{2} \rceil, \lceil \frac{g}{2} \rceil )$ by attaching $n - \lceil \frac{3g}{2} \rceil + 1$ pendant edges to a unique vertex.
\end{theorem}

\begin{proof}
We divide the proof into two cases when $g$ is even or $g$ is odd.
	
\textbf{Case 1.} Let $g$ be even. Assume that $g = 2a$ for some integer $a$. Let $B(G^*) = C(p,q,r)$, where $p+q = g = 2a$ and $p \le q \le r$. It is sufficient to prove that $p = q = r = a$. Note that $p+ q+r \ge 3a$ and $n \ge (p+q+r)-1$, that is, $n \ge 3a-1$.
	
If $n = 3a-1$, then $G^*$ does not have any pendant vertices and $p+q+r = 3a = g+a$. Hence, by using the relations $p + q = g$ and $p \le q \le r$, we have $p = q= r = a$.

If $n = 3a$, then $p+q+r \le 3a+1$, which implies that $r \le a+1$. Hence, $(p,q,r) \in \{(a-1,a+1,a+1),(a,a,a),(a,a,a+1)\}$. If $(p,q,r) = (a-1,a+1,a+1)$ or $(a,a,a+1)$, then $G^*$ does not have any pendant vertices.
	\begin{itemize}
		\item[(i)] If $a = 2$, then $(p,q,r) = \{(2,2,2),(2,2,3),(1,3,3)\}$. By Lemma~\ref{lem:subdiv} and Lemma~\ref{lem:subgraph}, it follows that $\rho_\alpha(C(2,2,3))<\rho_\alpha(C(2,2,2)) < \rho_\alpha(C_{2,2,2}^1)$, a contradiction. Again, By direct computations, we can find that $\rho_\alpha(C(1,3,3))=\frac{3\alpha+2+\sqrt{9\alpha^2-16\alpha+8}}{2}$
   and the $A_\alpha$-spectral radius of $C_{2,2,2}^1$ is the largest root of the polynomial 
   \begin{equation*}
    f(x)=x^4-10\alpha x^3+(28\alpha^2+14\alpha-7)x^2-(18\alpha^3+64\alpha^2-32\alpha)x+42\alpha^3-9\alpha^2-12\alpha+3.
   \end{equation*}
   Now, one can check that
   \begin{align*}
    & f\big(\rho_\alpha(C(1,3,3))\big)= f\bigg(\frac{3\alpha+2+\sqrt{9\alpha^2-16\alpha+8}}{2}\bigg)\\
    &=\frac{1}{2}\Big((9\alpha^4-20\alpha^3+11\alpha^2+2\alpha-2)+(3\alpha^3-4\alpha^2+3\alpha-2)\sqrt{9\alpha^2-16\alpha+8}\Big)\\
    &=\frac{1}{2}(\alpha-1)\Big((\alpha-1)(9\alpha^2-2\alpha-2)+(3\alpha^2-\alpha+2)\sqrt{9\alpha^2-16\alpha+8}\Big)\\
    &<0,
   \end{align*}
which implies that  $\rho_\alpha(C(1,3,3)) < \rho_\alpha(C_{2,2,2}^1)$, a contradiction. 
	\item[(iii)] If $a \ge 3$, then $G^*$ cannot have an edge between the two vertices of degree $3$ in $C(p,q,r)$. Hence, by Lemma~\ref{lem:d(u)-d(v)}, we have 
 \begin{align*}
  \rho_\alpha(G^*) < \max_{u \sim v} \{\alpha d(u) + (1-\alpha) d(v)\} = \max \{2+\alpha, 3-\alpha\}= 2+\alpha ~~ \text{for} ~~\frac{1}{2} \le \alpha < 1.  
 \end{align*}

    On the other hand, the maximum degree of $C_{a,a,a}^1$ is $4$ and hence by Lemma~\ref{lem:spec-Delta}, it follows that $\rho_\alpha(C_{a,a,a}^1)\geq 4 \alpha + \frac{(1-\alpha)^2}{\alpha}\geq 2 + \alpha$ for $\alpha \in[\frac{1}{2},1)$. Thus, $\rho_\alpha(G^*) < \rho_\alpha(C_{a,a,a}^1)$, a contradiction.
\end{itemize}
 Hence, $p = q = r = a$ for $n=3a$.

	If $n \ge 3a+1$ and $r = a$, then using the fact $p+q = 2a$ and $p \le q \le r$, we have $p = q = a$. Let $k = n -(p+q+r) +1$ be the number of pendant vertices in $G^*$. If $r \ge a+1$, then $k \le n-3a$. Note that, for a fixed value of $k$, $\max\{\alpha d(u)+(1-\alpha) m(u)\}$ is attained when the two degree $3$ vertices are adjacent and $k$ pendant edges are attached to a vertex of degree $3$ in $C(p,q,r)$. Since $G^*$ is not regular, by Lemma \ref{lem:d(u)-m(u)}, we have
	\begin{equation*}
		\rho_\alpha(G^*) < \max_{u} \{\alpha d(u)+(1-\alpha) m(u)\} = \alpha (k+3) + (1-\alpha) \frac{k+7}{k+3}.
	\end{equation*} 
 Observe that $f(x) = \alpha (x+3) + (1-\alpha) \dfrac{x+7}{x+3}$ is an increasing function in $x$ for $x \ge 0$ and $\frac{1}{2} \le \alpha < 1$. Therefore,
	\begin{align*}
		\rho_\alpha(G^*) < \alpha (n-3a+3) + (1-\alpha) \frac{n-3a+7}{n-3a+3} & = \alpha (n-3a+3) + (1-\alpha) + (1-\alpha) \frac{4}{n-3a+3}\\
        &= 1 + \frac{4}{n-3a+3} + \alpha \left(n-3a+2 - \frac{4}{n-3a+3}\right). 
	\end{align*} 
 Now, we claim that 
 $$1 + \frac{4}{n-3a+3} + \alpha \left(n-3a+2 - \frac{4}{n-3a+3}\right) \le \alpha(n-3a+4) + \frac{(1-\alpha)^2}{\alpha},$$ which is equivalent to $$\frac{4\alpha(1-\alpha)}{ 3\alpha^2 -3\alpha+1} \le n-3a+3.$$ 
 Note that $g(\alpha) = \frac{4\alpha(1-\alpha)}{ 3\alpha^2 -3\alpha+1}$ is a decreasing function of $\alpha$ for $\frac{1}{2} \le \alpha < 1$. Therefore, $g(\alpha)\leq g(\frac{1}{2})=4$.  Since $n \ge 3a+1$, we have $4 \le n-3a+3$ and hence the claim is true. Now, by Lemma~\ref{lem:spec-Delta}, we have
 \begin{align*}
 \rho_\alpha \big(C_{a,a,a}^{n-3a+1}\big)\geq \alpha \Delta \big(C_{a,a,a}^{n-3a+1}\big) + \frac{(1-\alpha)^2}{\alpha}=\alpha(n-3a+4)  + \frac{(1-\alpha)^2}{\alpha} ~~ \text{for}~~\frac{1}{2}\leq \alpha <1.    
 \end{align*}
  Thus, $\rho_\alpha(G^*) < \rho_\alpha(C_{a,a,a}^{n-3a+1})$, which is a contradiction.\\
	
\textbf{Case 2.} Let $g$ be odd. Assume that $g = 2a+1$ for some integer $a$. Let $B(G^*) = C(p,q,r)$, where $p+q = g = 2a+1$ and $p \le q \le r$. It is sufficient to prove that $p = a$ and $q = r = a+1$. Note that $p+ q+r \ge 3a+2$ and $n \ge (p+q+r)-1 = 3a+1$.
	
	If $n = 3a+1$, then $G^*$ does not contain any pendant vertices, and $p+q+r = 3a+2$. Since $p+q = 2a+1$, $r = a+1$ and $p \le q \le r$, we must have $p = a$ and $q = r = a+1$.
	
	If $n = 3a+2$, then $p+q+r \le 3a+3$. Since $p+q = g$ and $r \le a+2$, we have $q \le a+2$ and $p \ge a-1$ for $a \ge 2$. Hence, 
 $$(p,q,r) \in \{(a,a+1,a+1),(a,a+1,a+2),(a-1,a+2,a+2)\}.$$ 
 If $(p,q,r) = (a,a+1,a+2)$ or $(a-1,a+2,a+2)$, then $G^*$ does not have any pendant vertices. 
	\begin{itemize}
		\item[(i)] If $a = 1$, then $(p,q,r) \in \{(1,2,2),(1,2,3)\}$. By Lemma~\ref{lem:subdiv}, it follows that $\rho_\alpha(C(1,2,3)) < \rho_\alpha(C(1,2,2))$ and by Lemma~\ref{lem:subgraph}, we have $\rho_\alpha(C(1,2,2)) < \rho_\alpha(C_{1,2,2}^1)$, a contradiction.
		\item[(ii)] If $a = 2$, then $(p,q,r) \in \{(2,3,4),(1,4,4)\}$. By Lemma~\ref{lem:subdiv} and Lemma~\ref{lem:subgraph}, it follows that $\rho_\alpha(C(2,3,4)) <\rho_\alpha(C(2,3,3)) < \rho_\alpha(C_{2,3,3}^1)$, a contradiction. Again, by Lemma~\ref{lem:subdiv}, we have $\rho_\alpha(C(1,4,4)) <\rho_\alpha(C(1,3,3))$. Now, by direct computation, we can find that the $A_\alpha$-spectral radius of $C_{2,3,3}^1$ is the largest root of the polynomial
    \begin{align}
     g(x)&=x^6-14\alpha x^5+(71\alpha^2+16\alpha-8)x^4-(160\alpha^3+140\alpha^2-70\alpha)x^3+(160\alpha^4+380\alpha^3 \nonumber\\
     &~~~-134\alpha^2-56\alpha+14)x^2-(56\alpha^5+392\alpha^4-26\alpha^3-160\alpha^2+30\alpha+4)x+126\alpha^5 \nonumber\\
    &~~~+59\alpha^4-104\alpha^3+6\alpha^2+10\alpha-1. \label{eqn1}   
    \end{align}
    Now, one can verify that
     \begin{align*}
    & g\big(\rho_\alpha(C(1,3,3))\big)= g\bigg(\frac{3\alpha+2+\sqrt{9\alpha^2-16\alpha+8}}{2}\bigg)\\
    &=\frac{1}{2}\Big((30\alpha^6-132\alpha^5+249\alpha^4-250\alpha^3+133\alpha^2-30\alpha)\\
     &~~~~~~~~+(10\alpha^5-32\alpha^4+413\alpha^3-26\alpha^2+9\alpha-2)\sqrt{9\alpha^2-16\alpha+8}\Big)\\
    &=\frac{1}{2}(\alpha-1)\Big(\alpha(30\alpha^4-102\alpha^3+147\alpha^2-103\alpha+30)\\
     &~~~~~~~~~~~~~~~~~+(10\alpha^4-22\alpha^3+19\alpha^2-7\alpha+2)\sqrt{9\alpha^2-16\alpha+8}\Big)\\
    &<0,
   \end{align*}
which implies that  $\rho_\alpha(C(1,3,3)) < \rho_\alpha(C_{2,3,3}^1)$. Thus, we have $\rho_\alpha(C(1,4,4)) <\rho_\alpha(C(1,3,3)) < \rho_\alpha(C_{2,3,3}^1)$, a contradiction.
	\item[(iii)] If $a \ge 3$, then $G^*$ cannot have an edge between the two vertices of degree $3$ in $C(p,q,r)$. Thus, using Lemma~\ref{lem:d(u)-d(v)} and the fact $\frac{1}{2} \le \alpha < 1$, we have $$\rho_\alpha(G^*) < \max_{u \sim v}\{\alpha d(u) + (1-\alpha) d(v)\} = \max \{2+\alpha,3-\alpha\} = 2+\alpha \le 4 \alpha + \frac{(1-\alpha)^2}{\alpha}.$$
  Since $C_{a,a+1,a+1}^1$  has maximum degree $\Delta=4$, by Lemma~\ref{lem:spec-Delta}, we have $$4 \alpha + \frac{(1-\alpha)^2}{\alpha} =  \alpha \Delta + \frac{(1-\alpha)^2}{\alpha} \le  \rho_\alpha(C_{a,a+1,a+1}^1).$$ 
  Hence, $\rho_\alpha(G^*) <  \rho_\alpha(C_{a,a+1,a+1}^1)$, a contradiction.
\end{itemize}
Thus, for $n = 3a+2$, we have $p = a$ and $q = r = a+1$.

Finally, we consider the case when $n \ge 3a+3$. If $r = a+1$, then using the fact $p+q = g= 2a+1$ and $p \le q \le r$ we have $p = a-1$ and $q = a+1$. Next, we assume that $r \ge a+2$ and $k = n- (p+q+r) +1$ be the number of pendant vertices in $G^*$, which implies that $k \le n-(3a+2)$. Note that the quantity $\alpha d(u) + (1-\alpha) m(u)$ attains its maximum when $p = 1$ and $k$ pendent edges are incident to a degree $3$ vertex of $C(p,q,r)$. Thus, for a fixed $k$, by Lemma~\ref{lem:d(u)-m(u)}, we have $$\rho_\alpha(G^*) < \max_{u} \{\alpha d(u) + (1-\alpha) m(u)\} = \alpha (k+3) + (1-\alpha) \frac{k+7}{k+3},$$ since $G^*$ is not regular. Since $f(x) = \alpha (x+3) + (1-\alpha) \dfrac{x+7}{x+3}$ is an increasing function in $x$ for $x \ge 0$ and $\frac{1}{2} \le \alpha < 1$, we have 
 \begin{align*}
    \rho_\alpha(G^*) < \alpha(n-3a+1) + (1-\alpha) \frac{n-3a+5}{n-3a+1} &= \alpha(n-3a+1) + (1-\alpha) + (1-\alpha)\frac{4}{n-3a+1}\\ 
    &= 1+\frac{4}{n-3a+1} + \alpha \left(n-3a - \frac{4}{n-3a+1}\right)
 \end{align*}
Now, we claim that $$1+\frac{4}{n-3a+1} + \alpha \left(n-3a - \frac{4}{n-3a+1}\right) \le \alpha(n-3a+2) + \frac{(1-\alpha)^2}{\alpha},$$ which is equivalent to prove the inequality $$\frac{4\alpha(1-\alpha)}{ 3\alpha^2 -3\alpha+1}  \le n-3a+1.$$ Observe that $g(\alpha) = \frac{4\alpha(1-\alpha)}{ 3\alpha^2 -3\alpha+1} $ is a decreasing function of $\alpha$ for $\frac{1}{2} \le \alpha < 1$ and hence $g(\alpha)\leq g(\frac{1}{2})=4$. Since $n \ge 3a+3$, $4 \le n-3a+1$ and hence our claim is true. Now, for $\frac{1}{2}\leq \alpha <1$, by Lemma~\ref{lem:spec-Delta}, it follows that 
$$\rho_\alpha(G^*) < \alpha(n-3a+2) + \frac{(1-\alpha)^2}{\alpha}  = \alpha \Delta \big(C_{a,a+1,a+1}^{n-3a-1}\big) + \frac{(1-\alpha)^2}{\alpha} \le \rho_\alpha \big(C_{a,a+1,a+1}^{n-3a-1}\big),$$
which is a contradiction. This completes the proof.
\end{proof}

\begin{theorem}\label{thm:bicyclic-max}
	Let $n,g$ be integers such that $\lceil \frac{3g}{2} \rceil - 1 \le n$ and $G^*$ be the graph with maximal $A_\alpha$-spectral radius in $\mathcal{B}_n^2(g)$ for $\frac{1}{2}\le \alpha <1$. Then  $G^* \cong C_{\lfloor \frac{g}{2} \rfloor,\lceil \frac{g}{2} \rceil,\lceil \frac{g}{2} \rceil}^{n - \lceil \frac{3g}{2} \rceil+1}$.
\end{theorem}

\begin{proof}
By Theorem~\ref{thm:bi-2-unique}, it follows that the maximal graph $G^*$ is obtained by attaching $n - \lceil \frac{3g}{2} \rceil+1$ pendent edges to a unique vertex, say $u$ of $C(\lfloor \frac{g}{2} \rfloor,\lceil \frac{g}{2} \rceil,\lceil \frac{g}{2} \rceil)$. Now, we have to show that $u$ is a vertex of degree $3$ in $C(\lfloor \frac{g}{2} \rfloor,\lceil \frac{g}{2} \rceil,\lceil \frac{g}{2} \rceil)$. Suppose on the contrary, we assume that $u$ is a vertex of degree $2$ in $C(\lfloor \frac{g}{2} \rfloor,\lceil \frac{g}{2} \rceil,\lceil \frac{g}{2} \rceil)$. For convenience, let $k = n - \lceil \frac{3g}{2} \rceil+1$. If $k = 0$, clearly $G^* \cong C_{\lfloor \frac{g}{2} \rfloor,\lceil \frac{g}{2} \rceil,\lceil \frac{g}{2} \rceil}^0$. Now, consider the following two cases:\\

\begin{figure}[ht]
	\centering
	\begin{tikzpicture}[scale=0.5]
		\draw[fill=black] (1,-2) circle (2pt);
		\draw[fill=black] (-1,0) circle (2pt);
		\draw[fill=black] (1,2) circle (2pt);
		\draw[fill=black] (3,0) circle (2pt);
		\draw[fill=black] (1,4) circle (2pt);

		\draw[thick] (3,0) -- (1,2) -- (-1,0);
		\draw[thick] (3,0) -- (1,-2) -- (-1,0);
		\draw[thick] (-1,0) -- (1,0) -- (3,0);
		\draw[thick] (1,2) -- (1,4);
		
		\node at (1,-3) {$G_1$};
	\end{tikzpicture}\hspace{0.5cm}
	\begin{tikzpicture}[scale=0.5]
	\draw[fill=black] (1,-2) circle (2pt);
	\draw[fill=black] (-1,0) circle (2pt);
	\draw[fill=black] (1,2) circle (2pt);
	\draw[fill=black] (3,0) circle (2pt);
	\draw[fill=black] (1,4) circle (2pt);
	\draw[fill=black] (1,0) circle (2pt);
	
	\draw[thick] (3,0) -- (1,2) -- (-1,0);
	\draw[thick] (3,0) -- (1,-2) -- (-1,0);
	\draw[thick] (-1,0) -- (1,0) -- (3,0);
	\draw[thick] (1,2) -- (1,4);
	
	\node at (1,-3) {$G_2$};
	\end{tikzpicture}\hspace{0.5cm}
	\begin{tikzpicture}[scale=0.5]
		\draw[fill=black] (-1,-2) circle (2pt);
		\draw[fill=black] (1,-2) circle (2pt);
		\draw[fill=black] (-3,0) circle (2pt);
		\draw[fill=black] (-1,0) circle (2pt);
		\draw[fill=black] (0,2) circle (2pt);
		\draw[fill=black] (3,0) circle (2pt);
		\draw[fill=black] (0,4) circle (2pt);
		\draw[fill=black] (1,0) circle (2pt);
		
		\draw[thick] (-3,0) -- (0,2) -- (3,0) -- (1,-2) -- (-1,-2) -- (-3,0);
		\draw[thick] (-3,0) -- (-1,0) -- (1,0) -- (3,0);
		\draw[thick] (0,2) -- (0,4);
		
		\node at (0,-3) {$G_3$};
	\end{tikzpicture}\hspace{0.5cm}
	\begin{tikzpicture}[scale=0.5]
		\draw[fill=black] (-1,-2) circle (2pt);
		\draw[fill=black] (1,-2) circle (2pt);
		\draw[fill=black] (-3,0) circle (2pt);
		\draw[fill=black] (-1,0) circle (2pt);
		\draw[fill=black] (0,2) circle (2pt);
		\draw[fill=black] (3,0) circle (2pt);
		\draw[fill=black] (-1,2) circle (2pt);
		\draw[fill=black] (1,0) circle (2pt);
		
		\draw[thick] (-3,0) -- (0,2) -- (3,0) -- (1,-2) -- (-1,-2) -- (-3,0);
		\draw[thick] (-3,0) -- (-1,0) -- (1,0) -- (3,0);
		\draw[thick] (-1,0) -- (-1,2);
		
		\node at (0,-3) {$G_4$};
	\end{tikzpicture}\hspace{0.5cm}
	\begin{tikzpicture}[scale=0.5]
		\draw[fill=black] (-1,-2) circle (2pt);
		\draw[fill=black] (1,-2) circle (2pt);
		\draw[fill=black] (-1,2) circle (2pt);
		\draw[fill=black] (-1,0) circle (2pt);
		\draw[fill=black] (1,2) circle (2pt);
		\draw[fill=black] (3,0) circle (2pt);
		\draw[fill=black] (5,0) circle (2pt);
		\draw[fill=black] (-1,2) circle (2pt);
		\draw[fill=black] (1,0) circle (2pt);
		
		\draw[thick] (-1,2) -- (1,2) -- (3,0) -- (5,0);
		\draw[thick] (-1,0) -- (1,0) -- (3,0) -- (5,0);
		\draw[thick] (-1,-2) -- (1,-2) -- (3,0) -- (5,0);
		
		\node at (1,-3) {$G_5$};
	\end{tikzpicture}	
	\caption{Some special bicyclic graphs or subgraphs.} \label{fig:bicyclic0}
\end{figure}
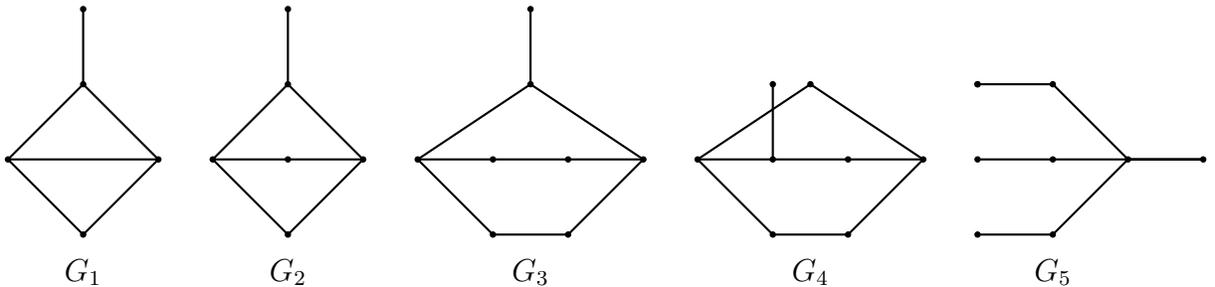

	\textbf{Case 1.} Let $k = 1$. If $g \in \{3,4,5\}$, then $G^*\cong G_1,G_2,G_3$ or $G_4$. By Lemma \ref{lem:coal}, it  follows that $\rho_\alpha(G_1)<\rho_\alpha(C_{1,2,2}^1)$ and $\rho_\alpha(G_2)<\rho_\alpha(C_{1,2,2}^1)$, a contradiction. Again, by Lemma \ref{lem:d(u)-m(u)}, we have  $\rho_\alpha(G_3),\rho_\alpha(G_4)=\max \big \{3-\alpha,2+\alpha,\frac{2\alpha+7}{3}\big \}=\frac{2\alpha+7}{3}$. Again, $\rho_\alpha(C_{2,3,3}^1)$ is the largest root of $g(x)=0$, where $g(x)$ is given in equation (\ref{eqn1}). Now, one can verify  that $g\big(\frac{2\alpha+7}{3}\big)=\frac{1}{729}\big(-992\alpha^6+22230\alpha^5-100473\alpha^4+183286\alpha^3-152901\alpha^2+56040\alpha-7190\big)<0$. Hence, $\rho_\alpha(G_3)<\rho_\alpha(C_{2,3,3}^1)$ and $\rho_\alpha(G_4)<\rho_\alpha(C_{2,3,3}^1)$, a contradiction.
   	
    Let $g\geq 6$. By Lemma \ref{lem:subdiv}, it follows that	any maximal graph $G^*$ on girth $g \ge 6$ has a smaller $A_\alpha$-spectral radius than that of the maximal graph with girth $g =5$, that is, $\rho_\alpha(G^*)<\rho_\alpha(G_3)$. On the other hand, for $g \ge 6$, $G_5$ is a subgraph of $C_{\lfloor \frac{g}{2} \rfloor,\lceil \frac{g}{2} \rceil,\lceil \frac{g}{2} \rceil}^1$ and hence by Lemma \ref{lem:subgraph}, we have $\rho_\alpha(G_5) < \rho_\alpha \big(C_{\lfloor \frac{g}{2} \rfloor,\lceil \frac{g}{2} \rceil,\lceil \frac{g}{2} \rceil}^1\big)$. By direct computations, one can find that $\rho_\alpha(G_5)$ is the largest root of $h_1(x)$ and $\rho_\alpha(G_3)$ is the largest root of $h_2(x)$, where 
    \begin{align*}
     h_1(x)&=x^4-8\alpha x^3+(16\alpha^2+10\alpha-5)x^2-(8\alpha^3+28\alpha^2-14\alpha)x+(14\alpha^3-3\alpha^2-4\alpha+1),\\
     h_2(x)&=x^4-(8\alpha+1)x^3+(17\alpha^2+17\alpha-5)x^2-(8\alpha^3+44\alpha^2-10\alpha-3)x+18\alpha^3+11\alpha^2\\
     &~~~-13\alpha+2.
    \end{align*}
    If $x_1=\rho_\alpha(G_3)$, then one can verify that $h_1(x_1)=h_1(x_1)-h_2(x_1)=x_1^3-(\alpha^2+7\alpha)x_1^2+(16\alpha^2+4\alpha-3)x_1-(4\alpha^3+14\alpha^2-9\alpha+1)<0$. Hence,  $\rho_\alpha(G_3)=x_1<\rho_\alpha(G_5)$. Thus, $\rho_\alpha(G^*)<\rho_\alpha(G_3)<\rho_\alpha(G_5) < \rho_\alpha \big(C_{\lfloor \frac{g}{2} \rfloor,\lceil \frac{g}{2} \rceil,\lceil \frac{g}{2} \rceil}^1\big)$, which is a contradiction. \\
	
	\textbf{Case 2.} Let $k \ge 2$. The maximum value of $\{\alpha d(u) + (1-\alpha)m(u)\}$ over all vertices $u \in V(G^*)$ is attained if $u$ is adjacent to the two vertices of degree $3$ in $C(\lfloor \frac{g}{2} \rfloor,\lceil \frac{g}{2} \rceil,\lceil \frac{g}{2} \rceil)$. Since $G^*$ is not regular, by  Lemma~\ref{lem:d(u)-m(u)}, we have 
    \begin{align*}
     \rho_\alpha(G^*) < \max_{u \in V(G^*)} \{\alpha d(u) + (1-\alpha)m(u)\} &= \alpha (k+2) + (1-\alpha) \frac{k+6}{k+2}\\ &= 1 + \alpha (k+1) + (1-\alpha)\frac{4}{k+2}.  
    \end{align*}
Now, we claim that $$1 + \alpha (k+1) + (1-\alpha)\frac{4}{k+2} \le \alpha (k+3) + \frac{(1-\alpha)^2}{\alpha},$$
which is equivalent to prove the inequality 
$$\frac{4\alpha(1-\alpha)}{3\alpha^2-3\alpha+1} \le k+2.$$
Observe that $g(\alpha) = \frac{4\alpha(1-\alpha)}{3\alpha^2-3\alpha+1}$ is a decreasing function of $\alpha$ for $\frac{1}{2} \le \alpha < 1$ and hence $g(\alpha)\leq g(\frac{1}{2})=4$. Since $k \ge 2$, the above inequality is true and hence the claim is also true. Note that $C_{\lfloor \frac{g}{2} \rfloor,\lceil \frac{g}{2} \rceil,\lceil \frac{g}{2} \rceil}^k$ has maximum degree $\Delta = (k+3)$ and hence by Lemma~\ref{lem:spec-Delta}, we have $$ \alpha (k+3) + \frac{(1-\alpha)^2}{\alpha} = \alpha \Delta + \frac{(1-\alpha)^2}{\alpha} \le \rho_\alpha(C_{\lfloor \frac{g}{2} \rfloor,\lceil \frac{g}{2} \rceil,\lceil \frac{g}{2} \rceil}^k).$$
Thus, we have $\rho_\alpha(G^*) < \rho_\alpha(C_{\lfloor \frac{g}{2} \rfloor,\lceil \frac{g}{2} \rceil,\lceil \frac{g}{2} \rceil}^k)$, which is a contradiction. 

Hence, $u$ is a vertex of degree $3$ in $C(\lfloor \frac{g}{2} \rfloor,\lceil \frac{g}{2} \rceil,\lceil \frac{g}{2} \rceil)$. This completes the proof.
\end{proof}

Now, we determine the unique graphs with maximum $A_\alpha$-spectral radius in $\mathcal{B}_n(g)$.

\begin{theorem}
Let $n,g$ be integers such that $\lceil \frac{3g}{2} \rceil - 1 \le n$. If $G\in \mathcal{B}_n(g)$ and $\alpha \in [\frac{1}{2},1)$, then 
$$\rho_\alpha(G) \leq \rho_\alpha \Big(C_{\lfloor \frac{g}{2} \rfloor,\lceil \frac{g}{2} \rceil,\lceil \frac{g}{2} \rceil}^{n - \lceil \frac{3g}{2} \rceil+1}\Big).$$
The equality holds if and only if $G\cong C_{\lfloor \frac{g}{2} \rfloor,\lceil \frac{g}{2} \rceil,\lceil \frac{g}{2} \rceil}^{n - \lceil \frac{3g}{2} \rceil+1}$.
\end{theorem}
\begin{proof}
For $\frac{1}{2}\le \alpha <1$, by Theorem \ref{thm:bicyclic-girth}, it follows that $\rho_\alpha(G) \leq \rho_\alpha(B_n^1(g))$ for $G\in \mathcal{B}_n^1(g)$ and by Theorem \ref{thm:bicyclic-max}, we have $\rho_\alpha(G) \leq \rho_\alpha \Big(C_{\lfloor \frac{g}{2} \rfloor,\lceil \frac{g}{2} \rceil,\lceil \frac{g}{2} \rceil}^{n - \lceil \frac{3g}{2} \rceil+1}\Big)$ for $G\in \mathcal{B}_n^2(g)$. Therefore, it is sufficient to show that $\rho_\alpha(B_n^1(g)) < \rho_\alpha \Big(C_{\lfloor \frac{g}{2} \rfloor,\lceil \frac{g}{2} \rceil,\lceil \frac{g}{2} \rceil}^{n - \lceil \frac{3g}{2} \rceil+1}\Big)$. Now, consider the following two cases:\\

\noindent \textbf{Case 1.} Let $g=3$. Then, by applying Lemma \ref{lem:edge}, we have $\rho_\alpha(B_n^1(3)) < \rho_\alpha (C_{1,2,2}^{n-4})$.\\

\noindent \textbf{Case 2.} Let $g \geq 4$. By Lemma \ref{lem:d(u)-m(u)}, we have
\begin{align*}
\rho_\alpha(B_n^1(g))& \leq \max \big \{\alpha d(u)+(1-\alpha)m(u):u\in V\big(B_n^1(g)\big)\big \}  \\
& \leq \alpha(n-2g+5)+(1-\alpha)\bigg(\frac{n-2g+9}{n-2g+5}\bigg)\\
& \leq \alpha(n-2g+3)+2.
\end{align*}
Again, for $\frac{1}{2}\le \alpha <1$, by Lemma \ref{lem:spec-Delta}, we have
$$\rho_\alpha \Big(C_{\lfloor \frac{g}{2} \rfloor,\lceil \frac{g}{2} \rceil,\lceil \frac{g}{2} \rceil}^{n - \lceil \frac{3g}{2} \rceil+1}\Big)\geq \alpha \Delta + \frac{(1-\alpha)^2}{\alpha}=\alpha \Big(n -\Big \lceil \frac{3g}{2}\Big \rceil+4\Big) + \frac{(1-\alpha)^2}{\alpha}.$$
Since $g\geq 4$, $\alpha \big(n -\big \lceil \frac{3g}{2}\big \rceil+4\big) + \frac{(1-\alpha)^2}{\alpha}-\alpha(n-2g+3)-2=\alpha \big(\big \lfloor \frac{g}{2}\big \rfloor+1\big) + \frac{(1-\alpha)^2}{\alpha}-2\geq 0$ for  $\frac{1}{2}\le \alpha <1$. Hence,    
$$\rho_\alpha(B_n^1(g)) \leq \alpha(n-2g+3)+2 \leq \alpha \Big(n -\Big \lceil \frac{3g}{2}\Big \rceil+4\Big) + \frac{(1-\alpha)^2}{\alpha} \leq \rho_\alpha \Big(C_{\lfloor \frac{g}{2} \rfloor,\lceil \frac{g}{2} \rceil,\lceil \frac{g}{2} \rceil}^{n - \lceil \frac{3g}{2} \rceil+1}\Big).$$
This completes the proof.
\end{proof}

\subsection{Extremal bicyclic graphs in $\mathscr{B}_n(k)$}

In this subsection, we identify the bicyclic graphs with the maximum $A_\alpha$-spectral radius among $\mathscr{B}_n(k)$.
Let $\mathscr{B}_n^{1}(k)\in \mathscr{B}_n(k)$ such that the intersection of the two cycles is at most one vertex, and let $\mathscr{B}_n^{2}(k)\in \mathscr{B}_n(k)$ such that the intersection of the two cycles is at least one edge. Note that $\mathscr{B}_n(k) = \mathscr{B}_n^{1}(k) \cup \mathscr{B}_n^{2}(k)$. 

Let $B_n^1(p,t,q,k)$ be the set of all bicyclic graphs with $k$ pendant vertices and there is a path of length $t-1$  $(t \ge 1)$ between the two cycles of length $p,q$. Let $B_n^2(p,t,q,k)$ be the set of all bicyclic graphs with $k$ pendant vertices containing $C(p,t,q)$. Note that $B_n^1(p,t,q,k) \subset \mathscr{B}_n^{1}(k)$ and $B_n^2(p,t,q,k) \subset \mathscr{B}_n^{2}(k)$.

 Let $\mathfrak{B}_n^1(k)$ be the set of all bicyclic graphs containing two triangles with exactly one common vertex and $k$ paths of almost equal lengths are incident to that common vertex. In the following theorem, we prove that the graph $\mathfrak{B}_n^1(k)$ has maximum $A_\alpha$-spectral radius in $\mathscr{B}_n^{1}(k)$.

\begin{theorem}\label{thm:bicyclic1}
Let $n,k$ be integers with $1 \le k \le n-3$. If $G \in  \mathscr{B}_n^{1}(k)$, then
 $$\rho_\alpha(G) \leq \rho_\alpha(\mathfrak{B}_n^1(k))$$ 
 with equality if and only if $G \cong \mathfrak{B}_n^1(k)$.
\end{theorem}
\begin{proof}
Let $G \in \mathscr{B}_n^{1}(k)$ be a graph such that the 
$A_\alpha$-spectral radius of $G$ is as large as possible. Then $G \in B_n^1(p,t,q,k)$ for some $t \ge 1$. Let the two cycles in $G$ be $C_p$ and $C_q$ of lengths $p$ and $q$, respectively. From the proof of Theorem~\ref{thm:bicyclic-girth}, we can conclude that the maximal graph $G$ has two cycles $C_p$ and $C_q$ which share a common vertex, say $v_1$ and there are $k$ paths attached to the vertex $v_1$. Now, by Lemma~\ref{lem:path}, it follows that all $k$ paths attached to $v_1$ have almost equal lengths. Finally, by using the similar argument as in Lemma~\ref{lem:uni2}, it follows that $p = q = 3$. Thus, $G \cong \mathfrak{B}_n^1(k)$.
\end{proof}

 Let $\mathfrak{B}_n^2(k)$ be the class of bicyclic graphs containing two triangles which share an edge and $k$ paths of almost equal lengths are incident to a vertex of degree $3$. In the following theorem, we prove that the graph $\mathfrak{B}_n^2(k)$ has maximum $A_\alpha$-spectral radius in $\mathscr{B}_n^{2}(k)$.
 
\begin{theorem}\label{thm:bicyclic2}
Let $n,k$ be integers with $1 \le k \le n-3$. If $G \in  \mathscr{B}_n^{2}(k)$, then $$\rho_\alpha(G) \leq \rho_\alpha(\mathfrak{B}_n^2(k)) $$ 
and the equality holds if and only if $G \cong \mathfrak{B}_n^2(k)$.
\end{theorem}

\begin{proof}
Let $G \in \mathscr{B}_n^{2}(k)$ be a graph such that the 
$A_\alpha$-spectral radius of $G$ is as large as possible.  Then $G \in B_n^2(p,t,q,k)$ for some $t \ge 1$. Let $V(G) = \{v_1,v_2,\hdots,v_n\}$ be the vertex set of $G$ and let $\mathbf{x} = (x_1,\hdots,x_n)^t$ be the Perron vector of $A_\alpha(G)$, where $x_i$ corresponds to the vertex $v_i$. Let the two cycles in $G$ be $C_p$ and $C_q$ of lengths $p$ and $q$, respectively.
		
		Now, by proceeding similar to the proof of Theorem \ref{thm:bicyclic-girth} and using Lemma \ref{lem:path}, it follows that the maximal graph $G$ is a graph $C(p,t,q)$ in which $k$ paths of almost equal lengths are attached to a single vertex, say $v\in C(p,t,q)$. Now, we claim that $p,t,q \le 2$. On the contrary, if $p\geq 3$, then there exists an internal path $P_{l+1}=vv_1\hdots v_l$ of length $l\geq 2$ starting from $v$. Now, consider $G_1=G-vv_1-v_1v_2+v_1v_p$, where $v_p$ is a pendant vertex in $G$. Then, by Lemma \ref{lem:subdiv} and Lemma \ref{lem:subgraph}, we have $\rho_\alpha(G) < \rho_\alpha(G_1)$, which is a contradiction. Thus, $p,t,q \le 2$. Again, by Lemma \ref{lem:subdiv}, it follows that at most one of $p,t,q$ can be equal to $1$. Without loss of generality we assume that $t = 1$. Thus, the maximal graph $G$ is the graph obtained from $C(2,1,2)$ by attaching $k$ paths of almost equal lengths to a single vertex $v$ of $C(2,1,2)$. Now, by Lemma \ref{lem:coal}, it follows that $d(v)=3$. Hence, $G \cong \mathfrak{B}_n^2(k)$. 
\end{proof}

Now, we identify the unique graphs with maximum $A_\alpha$-spectral radius in $\mathscr{B}_n(k)$.

\begin{theorem}
Let $n,k$ be integers with $1 \le k \le n-3$. If $G \in  \mathscr{B}_n(k)$, then $$\rho_\alpha(G) \leq \rho_\alpha(\mathfrak{B}_n^2(k)). $$ 
The equality holds if and only if $G \cong \mathfrak{B}_n^2(k)$.
\end{theorem} 
\begin{proof}
By Theorem \ref{thm:bicyclic1} and  Theorem \ref{thm:bicyclic2}, it is sufficient to prove that $\rho_\alpha(\mathfrak{B}_n^1(k)) < \rho_\alpha(\mathfrak{B}_n^2(k))$. Now, by using Lemma \ref{lem:edge}, it follows that $\rho_\alpha(\mathfrak{B}_n^1(k)) < \rho_\alpha(\mathfrak{B}_n^2(k))$. This completes the proof.   
\end{proof}

\section*{Acknowledgments}
Joyentanuj Das is partially supported by the National Science and Technology Council in Taiwan (Grant Id: NSTC-111-2628-M-110-002).
Iswar Mahato would like to thank Indian Institute of Technology Bombay for financial support through the Institute Post-doctoral Fellowship.

\small{

}

\end{document}